\documentclass[12pt]{amsart}

\usepackage{amsmath}
\usepackage[left=1in,right=1in,top=1in,bottom=1in]{geometry}
\usepackage{graphicx}
\usepackage{hyperref}

\newtheorem{theorem}{Theorem}[section]
\newtheorem{lemma}[theorem]{Lemma}

\numberwithin{equation}{section}
\numberwithin{theorem}{section}
\numberwithin{table}{section}
\numberwithin{figure}{section}

\newcommand{\T}{\mathcal T}

\newcommand{\pr}{\mathcal P}
\newcommand{\K}{\mathcal K}
\newcommand{\C}{\mathcal C}

\title[Solving a Generalization of Golomb's Nested Triangular Recursion Using Trees]
{A combinatorial approach for solving certain nested recursions with non-slow solutions}

\author[Isgur]{Abraham Isgur}
\address{Dept. of Mathematics, University of Toronto, CANADA}

\author[Kuznetsov]{Vitaly Kuznetsov}
\address{Dept. of Mathematics, University of Toronto, CANADA}

\author[Tanny]{Stephen M. Tanny}
\address{Dept. of Mathematics, University of Toronto, CANADA}

\begin{document}

\begin{abstract}
We define the generalized Golomb triangular recursion by
$g_{j,s,\lambda}(n) = g_{j,s,\lambda}(n - s - g_{j,s,\lambda}(n-j)) + \lambda j$.
For particular choices of the initial conditions, we show that the solution of the recursion is a non-slow monotone sequence for which we can provide a combinatorial interpretation in terms of a weighted count of the leaves of a certain labeled infinite tree. We discover that more than one such tree interpretation is possible, leading to different choices of the initial conditions and alternative solutions that are closely related. In the case $\lambda=1$ the initial conditions for these alternative tree interpretations coincide and we derive explicit closed forms for the solution sequence and its frequency function.
\end{abstract}


\maketitle

\section{Introduction} \label{sec1}

In this paper all values for the variables and parameters are integers unless otherwise specified.
It is shown in \cite{NonHomog} that for $k \geq 1$, $s \geq 0$, $j \geq 1$, $\nu$ any constant, and with specified initial conditions, the solution to the non-homogeneous Conolly recursion \begin{equation} \label{recursion} R(n) = \sum_{i=1}^kR(n-s-(i-1)j-R(n-ij))+ \nu \end{equation} has a fascinating combinatorial interpretation. Specifically, $R(n)$ counts the number of \emph{labels} that are less than or equal to $n$ in the leaves of a special labeled tree formed by ``grafting" infinitely many copies of a finite, rooted tree $\T$ (related to the parameter $\nu$) onto a labeled $k$-ary tree. The initial conditions for the recursion are specified by the finite tree $\T$ and the parameters $k,s$ and $j$.

Any solution to (\ref{recursion}) that can be characterized combinatorially in this way necessarily has the property that the difference $R(n+1) - R(n)$ between consecutive terms is either 0 or 1. We call such a solution sequence \emph{slowly growing} or, more briefly, \emph{slow} (see \cite{EIRT, Rpaper, NonHomog}).

In general, very little is known about the existence or behavior of non-slow solutions to (\ref{recursion}). The first example of such a solution appears in \cite{HiTan}. In \cite{CCT} it is shown that for $j=1$ the solution to the related homogeneous version of (\ref{recursion}) with $\nu=0$ is non-slow for all odd $k>1$, so long as the initial conditions are taken to be $k+s$ consecutive ones \footnote{A corresponding result is conjectured to hold for even $k$ when $s=0$; see \cite{CCT}, pp. 822-824.}.

To date all of the non-slow solutions to nested recursions of the form (\ref{recursion}) have been derived by applying purely analytical (as opposed to combinatorial) techniques. Further, no non-slow solutions are known to any recursions of this form that have $\nu \neq 0$.

In this paper we address both these deficiencies, thereby taking an initial step toward developing a combinatorial interpretation for non-slow solutions to non-homogeneous nested recursions. More precisely, we extend the tree-based solution methodology described in \cite{NonHomog} to derive a combinatorial interpretation for non-slow solutions to (\ref{recursion}) in the case $k=1$ and $\nu=\lambda j$, where $\lambda>0$ is a fixed parameter, that is, the family of recursions
\begin{align}
g_{j,s,\lambda}(n) = g_{j,s,\lambda}(n - s - g_{j,s,\lambda}(n-j)) + \lambda j
\label{recur}
\end{align}
with appropriately specified initial conditions.

Our focus on solving (\ref{recur}) is motivated by several considerations. First, observe that for $j=\lambda=1$, $s=0$ and initial condition $g_{1,0,1}(1)=1$, (\ref{recur}) is Golomb's non-homogeneous nested triangular recursion \cite{Golomb1990}. Golomb was the first to investigate a non-homogeneous nested recursion of this type; he showed that its solution is the slow sequence 1, 2, 2, 3, 3, 3, 4, 4, 4, 4, 5, 5, 5, 5, 5, \ldots where every positive integer $n$ appears precisely $n$ times. That is, the frequency function $\phi(n)$ of this sequence is given by $\phi(n) = n$. Golomb further observed that this solution also has the neat closed form \begin{equation} g_{1,0,1}(n) = \left\lfloor \frac{ \left\lfloor \sqrt{8n} \right\rfloor + 1}{2}\right\rfloor \label{closed} \end{equation}.

Second, one of us chanced upon the following natural related question: find a closed form for the general term in the sequence 1, 3, 3, 3, 5, 5, 5, 5, 5, \ldots where each odd natural number $2m+1$ appears $2m+1$ times. This non-slow, monotone sequence appears in \cite{sna} as entry A001650, where it is noted that it is generated by the ``Golomb-like" recursion coinciding with (\ref{recur}) with $j=2, s=0$, and $\lambda = 1$ and initial conditions 1, 3, 3. From there we successively generalized the form of the nested recursion along the lines of (\ref{recursion}) to yield (\ref{recur}), while simultaneously developing appropriate analogues for the initial conditions of interest (more details are provided on these in Section \ref{sec2}).

Finally, recursion (\ref{recur}) occurs naturally in the context of our analysis in \cite{NonHomog} for the case $k=1$, where it follows immediately from the formula for $\nu$ (see \cite{NonHomog}, p. 6) that with $k=1$ a tree-based slow solution can only exist if $\nu$ is a positive multiple of $j$ \footnote{In fact, this is true for the even more general tree-grafting methodology described in \cite{NonHomog}, Section 4. We are indebted to Mustazee Rahman for pointing this out.}. Here we solve this recursion with very different initial conditions from those in \cite{NonHomog}, in particular, they are not slow.

The outline of the rest of the paper is as follows. In the next section we show how to extend the tree-based solution methodology in \cite{NonHomog} to derive a non-slow solution to (\ref{recur}) with specified initial conditions. In Section \ref{sec3} we illustrate how more than one such tree interpretation is possible, leading to different choices of the initial conditions and alternative solutions that are closely related. In the case $\lambda=1$ the initial conditions for these alternative tree interpretations coincide, so we conclude in Section \ref{sec4} by deriving explicit closed forms for the solution sequence and its frequency function in this case.

\section{Combinatorial interpretation of the generalized Golomb triangular recursion} \label{sec2}

In this section we give a combinatorial interpretation of the sequence generated by (\ref{recur}); more precisely, we show how this sequence gives a weighted count of the leaves of the labeled infinite tree $\K$ which we construct now. Fix integers $j,\lambda > 0$ and $s \geq 0$, which will correspond to the desired parameters of the recursion. For ease of reference, in what follows we write $g(n)$ in place of $g_{j,s,\lambda}(n)$ where there is no confusion.

Our approach closely resembles that of \cite{NonHomog}. There are two key steps in our construction. First we construct the nodes and edges that form the skeleton of $\K$. Then we insert the consecutive positive integers that are the labels in the nodes of $\K$. To do this, first we must specify the order in which the nodes of $\K$ are to be traversed one at a time; then we insert the appropriate number of labels (either $1$ or $s$) in each node. To help describe this construction, we illustrate our discussion for the case $j=2$, $\lambda=3$ and $s=4$.

\subsection*{Constructing the skeleton of $\K$}
The skeleton of $\K$ consists of an infinite sequence $\K_i, i \geq 0$ of rooted, finite subtrees of $\K$ that we connect together to form $\K$.

We define $\K_0$ to be a chain consisting of two nodes connected by a single edge; we call this a chain of length $2$. We take one of the nodes to be the root of $\K_0$ while the other node at the end of the chain is a leaf. We form subtree $\K_1$ as follows: take a copy of $\K_0$ and attach exactly $\lambda$ chains with $j$ nodes (``length $j$") to the leaf node. Note that $\K_1$ has $\lambda$ leaves. For $i > 1$ we construct the subtree $\K_i$ in an analogous way: take a copy of $\K_{i-1}$ and attach a chain of length $j$ to each of its $\lambda$ leaves. For each $i$, we refer to the root of $\K_i$ as the $i^{th}$ \textbf{supernode} of $\K$. The single node connected to the $i^{th}$ supernode is called $i^{th}$ \textbf{knot node}, while all other nodes of $\K_i$ are \textbf{regular nodes}. Note that in every $\K_i, i \geq 1$ the regular nodes at the end of each of the $\lambda$ chains are leaves.

Next, for all $i \geq 0$ we connect the subtree $\K_i$ to $\K_{i+1}$ by adding an edge from the $i^{th}$ supernode to the $(i+1)^{th}$. To complete the construction of $\K$ we attach an extra regular node as a child of the first supernode in $\K_0$. We call this additional regular node the \textbf{initial leaf of $\K$}. See Figure \ref{fig1} for the construction of $\K$ for the case $j=2$, $\lambda=3$ and $s=4$; note that supernodes are marked by rectangles.

\begin{figure}[htpb]
\begin{center}
\includegraphics[scale=0.35]{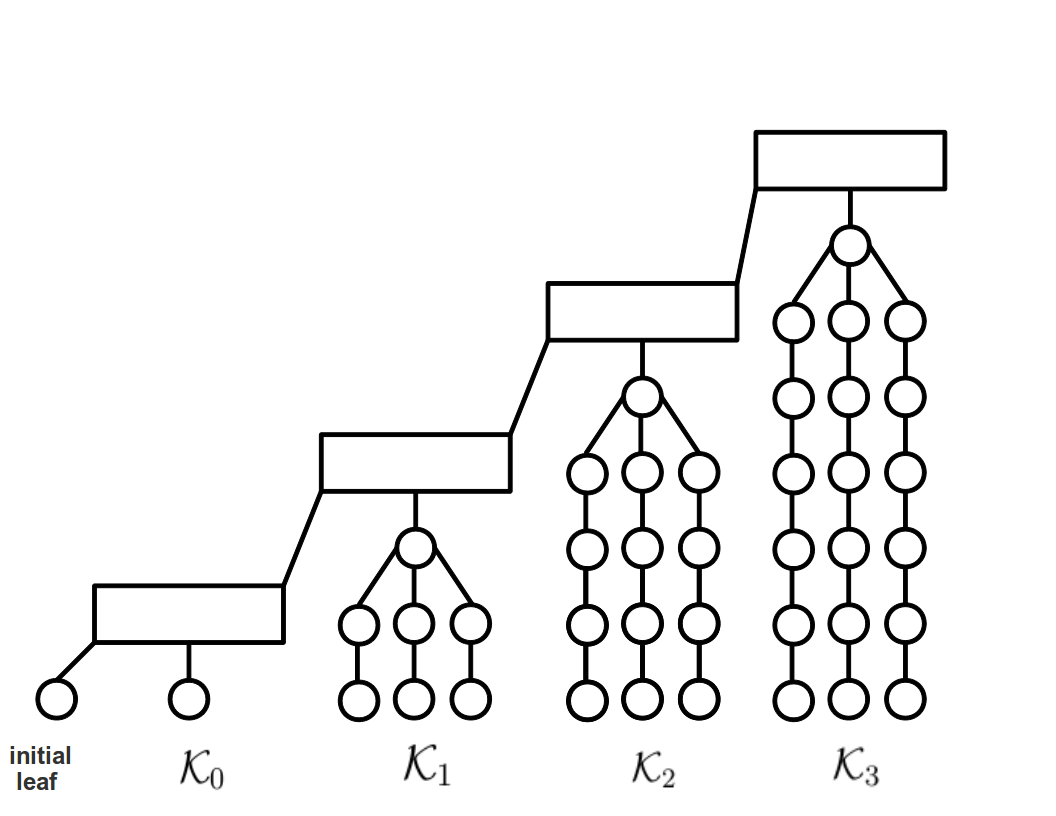}
\caption{The initial portion of the skeleton of $\K$} \label{fig1}
\end{center}
\end{figure}

\subsection*{Labeling $\K$}
We insert one label into each regular node of $\K$ and $s$ labels into each supernode. The labels consist of successive integers starting at 1. Before we can insert these labels we must specify the traversal order of the nodes in $\K$. We recursively define a \textbf{pre-order} traversal as follows: $\K_0$ is traversed by beginning at the initial leaf followed by the supernode itself and then to the remaining leaf. Having traversed $\K_i$ for $i \geq 0$, the subtree $\K_{i+1}$ is traversed next by starting at the $(i+1)^{th}$ supernode, then continuing to the knot node immediately below it, and then traversing the length of each of the $\lambda$ chains from left to right. See Figure \ref{fig2} for the labeling in case $j=2$, $\lambda=3$ and $s=4$.

\begin{figure}[htpb]
\begin{center}
\includegraphics[scale=0.35]{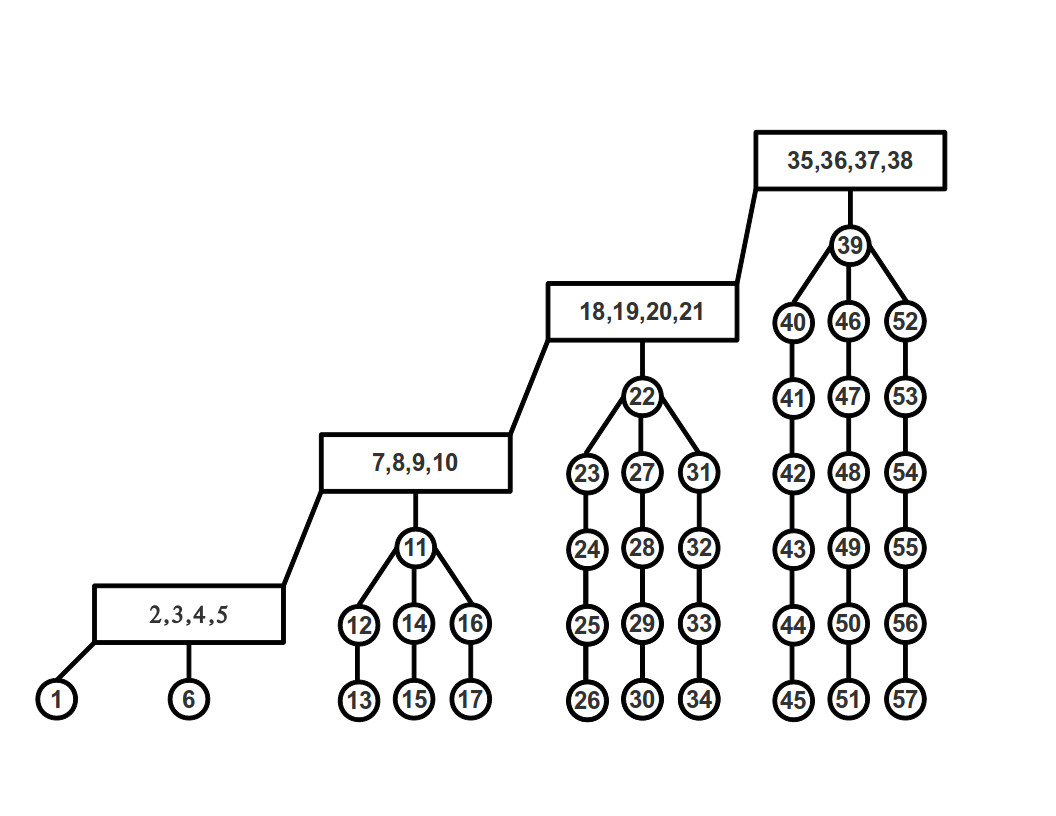}
\caption{Labeling the initial portion of the skeleton of $\K$} \label{fig2}
\end{center}
\end{figure}

Let $\K(n)$ be the subtree of $\K$ consisting of all the nodes of $\K$ with label(s) less than or equal to $n$. Observe that by definition $\K(n)$ necessarily consists of all the labeled nodes and edges in some finite number of subtrees $\K_0, \K_1, \ldots \K_{m-1}$ of $\K$, together with the subtree $\K^*_m$ which consists of a portion (including possibly all) of the subtree $\K_m$. For compactness we write $\K(n) = \K_0, \K_1, \ldots \K_{m-1}, \K^*_m$. If $\K^*_m$ is not all of $\K_m$ then it has fewer than $\lambda$ chains attached to its knot node and/or one of these chains has fewer than $mj$ nodes; in this case we call the subtree $\K^*_m$ \textbf{incomplete}.

We assign a weight of $1$ to the initial leaf of $\K$ while the rest of the leaves receive the weight $j$. Define the \textbf{leaf weight sequence} $w_{j,s,\lambda}(n)$ whose $n^{th}$ member equals the sum of the weights of the leaves of $\K$ that are in $\K(n)$. In other words, $w_{j,s,\lambda}(n)$ is the total weight of the leaves of $\K$ with label less than or equal to $n$; in Figure \ref{fig2} $w_{2,3,4}(n)$ begins 1,1,1,1,1,3,3,3,3,3,3,3,5,5,7,7,9,9,9,9,9,9,9,9,9,11, \ldots. \footnote{Observe that for $j=2, s=0, \lambda=1$ the leaf weight sequence of $\K$ is just the sequence 1,3,3,3,5,5,5,5,5, \ldots discussed in Section \ref{sec1}.} Where there is no confusion we write $w(n)=w_{j,s,\lambda}(n)$.

One needs to be careful to distinguish leaves of $\K$ from leaves of $\K(n)$. In particular, not all leaves of $\K(n)$ are leaves of $\K$. For example, in Figure \ref{fig2}, the node containing the label 52 is a leaf of $\K(52)$ but not of $\K$.

Now we are ready to state our main result on the combinatorial interpretation of the integer sequence generated by $g(n)$.

\begin{theorem} The leaf weight sequence $w(n)$ is the solution generated by the recursion (\ref{recur}) with initial conditions $g(n)= w(n)$ for $1\leq n \leq 3 +2s + \lambda j$. That is, if $g(n)=w(n)$ for $1\leq n \leq 3 +2s + \lambda j$ then $g(n)=w(n)$ for each positive integer $n$.
\label{thm:count}
\end{theorem}

The key to proving Theorem \ref{thm:count} lies with the pruning operation $\pr$ on $\K(n) = \K_0$, $\K_1$, $\ldots$, $\K_{m-1}$, $\K^*_m$ which yields a new tree $\pr\K(n)$ defined as follows: first disconnect the initial leaf from $\K_0$ and remove the rest of $\K_0$ (together with its labels) from the tree. Next, for $1 \leq i < m$, delete the $j$ largest labels along with the nodes that contain them from each chain connected to $i^{th}$ knot node in $\K_i$. From the construction of $\K$ it follows that doing so turns each $\K_i$ into $\K_{i-1}$, $1 \leq i < m$.

Now we turn our attention to $\K^*_m$. There are three cases: (1) if $\K^*_m$ has at least $2$ chains connected to its knot node, then delete the $j$ largest labels and the nodes that contain these labels from each chain with at least $j$ labels; (2) if $\K^*_m$ has fewer than 2 chains but has at least $j$ labels in total, then delete the largest $j$ labels and all the nodes that become unlabeled as a result; (3) finally, if $\K^*_m$ has fewer than $j$ labels, then don't do anything. In all cases, note that the pruning operation turns $\K^*_m$ into $\K^*_{m-1}$, some portion of $\K_{m-1}$.


Finally, we reconnect the initial leaf to the supernode of the subtree that started initially (before any pruning took place) as $\K_1$ and relabel the resulting tree in pre-order as defined above. It follows that $\pr \K(n)$ consists of the finite sequence of subtrees of $\K$: $\K_0, \K_1, \ldots \K^*_{m-1}$ and it must be the case that, after relabeling, $\pr \K(n) = \K(d)$ for some $d<n$. See Figure \ref{fig3} where we illustrate this for the pruning of $\K(52)$, which becomes the subtree $\K(31)$.

\begin{figure}[htpb]
\begin{center}
\includegraphics[scale=0.25]{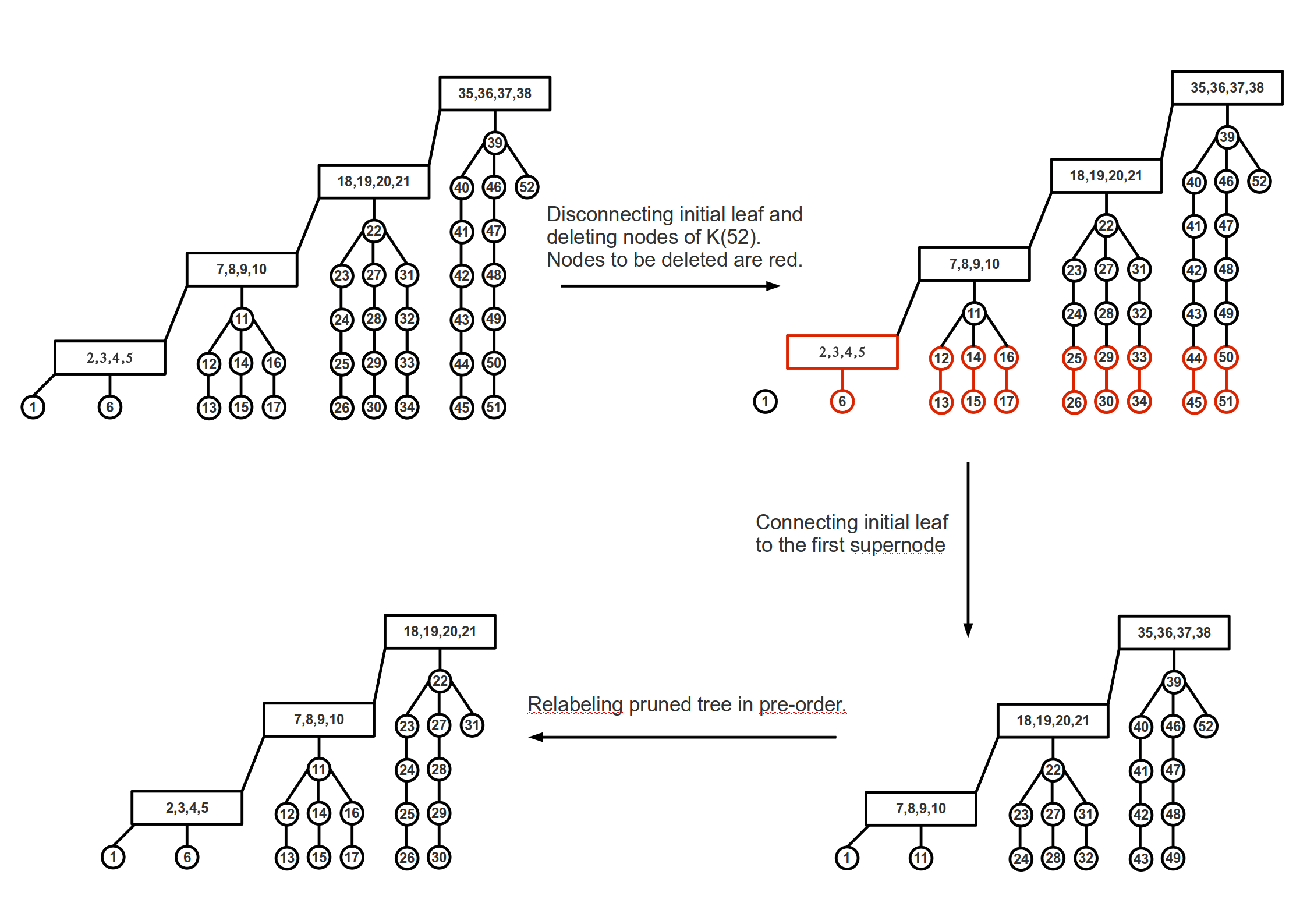}
\caption{Pruning $\K(52)$} \label{fig3}
\end{center}
\end{figure}

We now show the relation between $n$ and $d$ for any pruned subtree, from which Theorem \ref{thm:count} is immediate by a simple induction argument.

\begin{lemma}
For $n > 3+2s+\lambda j$, $\pr \K(n) = \K(n-s-w(n-j))$ and consequently $w(n) = w(n-s-w(n-j)) + \lambda j$.
\label{lemma:prun}
\end{lemma}
\begin{proof}
Fix $n > 3+2s+\lambda j$. Let $\K(n) =\K_0, \K_1, \ldots \K_{m-1}, \K^*_m$. Since $n> 3+2s+\lambda j$ and $3+2s+\lambda j$ is the total number of labels contained in $\K_0$ and $\K_1$, it follows that $m > 1$.

As we noted above, it must be the case that after relabeling $\pr \K(n) = \K(d)$ for some $d<n$. It suffices to show that $d=n-s-w(n-j)$, that is, we need to show that the total number of labels removed from the tree during the pruning operation is $s+w(n-j)$. By deleting the supernode of $\K_0$ we removed $s$ labels. Thus we must show that the number of labels removed from $\K_1, \ldots, \K_{m-1}, \K^*_m$ plus the one label removed by deleting the leaf of $\K_0$ together sum up to $w(n-j)$.

Note that since $\K_{m-1}$ is complete it has $\lambda$ chains of length $\geq j$ attached to its knot node. Denote the chain with the largest label in $\K_{m-1}$ by $\C_0$. Suppose $\K^*_m$ contains the $m^{th}$ knot node and there are at least 2 chains connected to it. Denote these chains in $\K^*_m$ by $\C_1, \ldots, \C_p$ where $ 2 \leq p \leq \lambda$. Otherwise, as we discussed above, we consider $\K^*_m$ itself as a chain, which we call $\C_1$. In either case, we know that the node with label $n$ lies in the chain $\C_p$ where $1 \leq p \leq \lambda$. We consider two cases.

Case 1: assume that $\C_p$ contains $t<j$ labels. Consider the node with label $n-j>(\lambda-1)j+2s+3>1$ in $\K(n)$. Since $\C_p$ has less than $j$ labels and the chain $\C_{p-1}$ has at least $j$ labels we know that the node with label $n-j$ lies in $\C_{p-1}$. Moreover, it is not a leaf node of $\K$ since a leaf node in this chain has the label $n-t$. Each chain in subtree $\K_i$ where $1 \leq i \leq m-1$ (except the last chain in $\K_{m-1}$) contributes $j$ to $w(n-j)$ which accounts for $j$ labels removed from it during pruning. Similarly, each of the chains $\C_i$ with $0 \leq i < p - 1$ contributes $j$ to $w(n-j)$ which accounts for $j$ labels removed from it. The single leaf of $\K_0$ also contributes $j$ to $w(n-j)$ which accounts for $j$ labels removed when pruning $\C_{p-1}$. Finally, the weight of initial leaf in $w(n-j)$ accounts for the single label removed from the leaf of $\K_0$. Therefore, $d=n-s-w(n-j)$ and $\pr \K(n)$ is same as $\K(n-s-w(n-j))$.

Case 2: assume that $\C_p$ contains $t \geq j$ labels. Again we consider the node with label $n-j$. In this case, this node is either the leaf of $\C_{p-1}$ or a node (but not the leaf) of $\C_p$. However, in both cases, $w(n-j)$ counts the weights of the same leaves. As before each chain in subtree $\K_i$ where $1 \leq i \leq m-1$ (except the last chain in $\K_{m-1}$) contributes $j$ to $w(n-j)$ which accounts for $j$ labels removed from it during pruning. Each of the chains $\C_i$ with $0 \leq i \leq p - 1$ contributes $j$ to $w(n-j)$ which accounts for $j$ labels removed from it. In this case, we have also removed $j$ labels from $\C_p$ while pruning $\K(n)$ and the weight of the single leaf of $\K_0$ contributes $j$ to $w(n-j)$ to account for this. Finally, the weight of initial leaf in $w(n-j)$ accounts for the single label removed from the leaf of $\K_0$ and we have that $d=n-s-w(n-j)$.

The weight of the leaves of $\K$ in $\pr \K(n)$ differs by $\lambda j$ from the weight of the leaves of $\K$ in $\K(n)$ since we have removed the leaf of $\K_0$ and $\lambda-1$ leaves of $\K_1$. This together with the fact that $\pr \K(n) = \K(n-s-w(n-j))$ imply that $w(n) = w(n-s-w(n-j)) + \lambda j$. This completes the proof of the lemma.

\end{proof}

\section{Alternative trees and initial conditions} \label{sec3}

In the preceding section we showed how to solve the recursion (\ref{recur}) in the case where the initial conditions are derived from the tree $\K$ and the leaf weighting scheme described there. It turns out that other trees and weighting schemes lead to a solution of (\ref{recur}) but with different initial conditions that depend upon the tree. For example, consider the initial portion of the tree $\K'$ in Figure \ref{fig4}.

\begin{figure}[htpb]
\begin{center}
\includegraphics[scale=0.35]{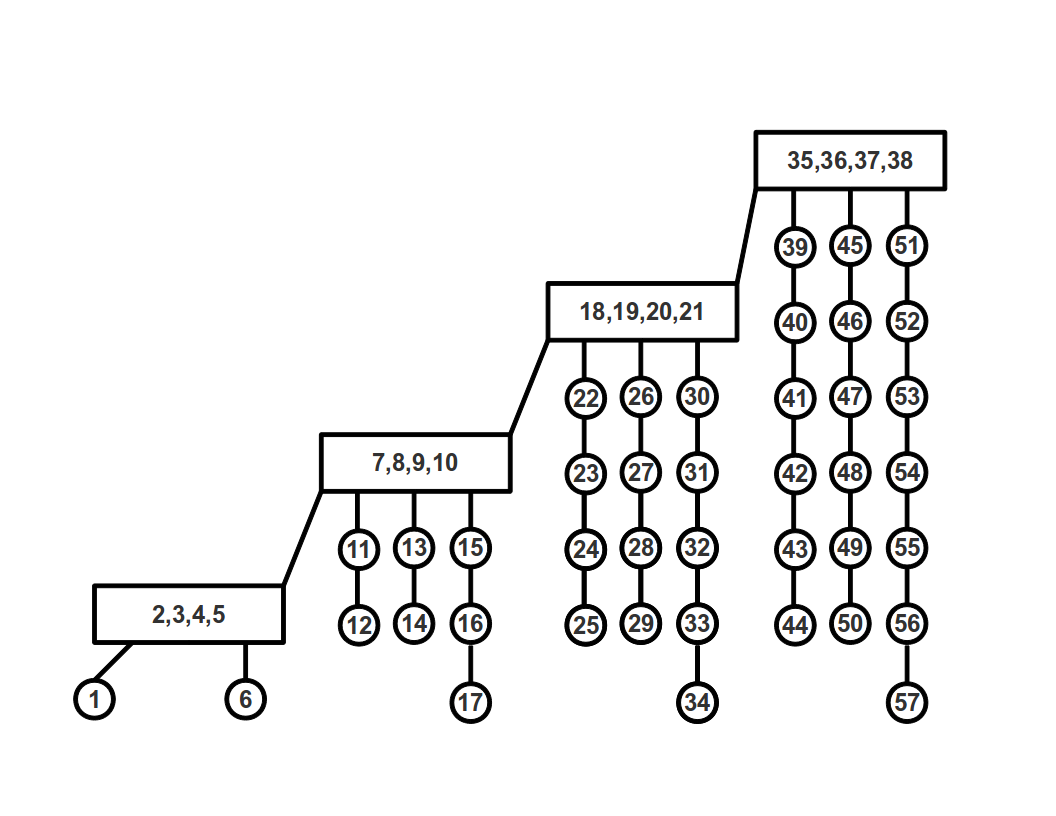}
\caption{Initial portion of an alternative tree $\K'$} \label{fig4}
\end{center}
\end{figure}

It is readily seen that $\K'$ is another tree that corresponds to (\ref{recur}) with $j=2$, $\lambda=3$ and $s=4$. Note that there is no knot node in $\K'$, but rather a ``tail" node in each of the sub-trees that make it up (in Figure \ref{fig4} these are the nodes 17, 34 and 57). Here the leaf weight sequence begins 1,1,1,1,1,3,3,3,3,3,3,5,5,7,7,7,9 \ldots. It is readily seen that this sequence, like the one related to Figure \ref{fig2}, consists of all the odd positive integers in order, but the frequencies with which these integers occur are different between the two sequences. Other tree constructions are also possible.

We can use the obvious generalization of the tree $\K'$ for arbitrary parameters $j, \lambda$ and $s$, together with the same approach discussed in Section \ref{sec2}, to demonstrate the analogue of Theorem \ref{thm:count} holds: if $g(n)=w'(n)$ for all $1\leq n \leq 3+\lambda j + 2s$, then $g(n)= w'(n)$ for all $n$, where $w'(n)$ is the leaf weight sequence for the tree $\K'$.

For the same choice of the parameters it is not hard to show, as we saw in the example above, that the members of the leaf weight sequences related to $\K$ and $\K'$ respectively are the same, although in general they appear with different frequencies. In the special case $\lambda=1$ the trees $\K$ and $\K'$ are the same, so the initial conditions and thus the leaf weight sequences generated by recursion (\ref{recur}) must match. In the next section we derive formulas related to the leaf weight sequence in this case.

\section{Closed-forms when $\lambda = 1$} \label{sec4}

For arbitrary $j>0, s\geq 0$ and $\lambda=1$, the tree $\K$ of Section \ref{sec2} can be described as follows: it consists of an infinite sequence of sub-trees $\K_i, i \geq 0$, where for $i>0$ each $\K_i$ is a chain of length $ij+2$ with root at the $i^{th}$ supernode and with $s+ij+1$ labels. We refer to the initial leaf of $\K$ as zeroth leaf. We call the only other leaf in $\K_0$ the first leaf of $\K$, while for $i>0$ the single leaf of $\K_i$ is the $(i+1)^{th}$ leaf of $\K$. In this section we derive closed forms for the leaf weight sequence sequence and its frequency function that generalize the results in \cite{Golomb1990} for the sequence $g_{1,0,1}(n)$. We begin with the frequency function.

\begin{theorem}
The recursion (\ref{recur}) with parameters $j,s,\lambda=1$ and initial conditions as in Theorem \ref{thm:count} generates a sequence $g(n)=g_{j,s,1}(n)$ with frequency function
\begin{align*}
\phi(n) =
\begin{cases}
n+s, \text{  } n \equiv 1 \bmod{j} \\
0, \text{  otherwise}
\end{cases}
\end{align*}
\label{thm:freq}
\end{theorem}

\begin{proof}
By Theorem \ref{thm:count} $g(n)=w(n)$, so it is enough to prove this statement for $w(n)$. Let $d$ be the label of the $m^{th}$ leaf in $\K$. If $m=0$ then $d=1$ and thus $w(d)=1=mj+1$. If $m\geq 1$, then $w(d)$ is the sum of the weights of the leaves of $\K$ in $\K_0, \ldots, \K_{m-1}$. Since each of these chains has a single leaf of weight $j$ and $\K_0$ also has an initial leaf with weight $1$, we have $w(d)=mj+1$.

If $t$ is the label of the $(m+1)^{th}$ leaf, then $w(n)=mj+1$ for $1<d\leq n < t$. By the construction of $\K$ there are exactly $s+mj+1$ labels $n$ such that $d\leq n < t$ (one leaf label $d$ and $s+mj$ labels in $\K_m$). Therefore, if $q \equiv 1 \bmod{j}$ then $q$ is repeated at least $q+s$ times in the sequence $w(n)$. But the sequence $w(n)$ is non-decreasing and $w(d-1)=(m-1)j+1$ (if $d>1$) and $w(t)=mj+j+1$. Thus we conclude that if $q \equiv 1 \bmod{j}$ then $q$ is repeated exactly $q+s$ times. This argument also shows that $w(n)\equiv 1 \bmod{j}$ for all $n>0$ and if $q \not\equiv 1 \bmod{j}$ then $\phi(q)=0$ which completes the proof.
\end{proof}

Theorem \ref{thm:freq} shows that for $j>0$, $s \geq 0$ and $\lambda=1$ recursion (\ref{recur}) generates the sequence $1^{s+1}, (j+1)^{s+j+1}, (2j+1)^{s+2j+1}, \ldots$, where $(mj+1)^{s+mj+1}$ means that the value $mj+1$ is repeated $s+mj+1$ times. For $j=1$ and $s=0$ this is the sequence is 1,2,2,3,3,3, \ldots with frequency function $\phi(n)=n$ discussed in \cite{Golomb1990}.

Next we generalize (\ref{closed}) by deriving a closed form solution for $g_{j,s,1}(n)$.

\begin{theorem}
Let $p_m = 1+\sum\limits_{i=0}^m(s+ij+1)$ and $F(n) = max\{p_m : p_m \leq n, \text{  } 0 \leq m \} \cup \{1\}$. Then for all positive integers $n$
\begin{align*}
g_{j,s,1}(n) = \frac{(j-2s) + \sqrt{(2s-j)^2 + 4(2jF(n) + 2s + 1 - 3j)}}{2}.
\end{align*}
\label{thm:closed}
\end{theorem}

\begin{proof}
To begin, observe that $F(n)$ is the label of the last leaf of $\K(n)$. We prove the formula in two steps. First we show that it holds when $n$ is a leaf label. Then we extend this formula to apply to all other natural numbers $n$.

When $n=1$ it is the label of the $0^{th}$ leaf. But then $F(1) = 1$ so the right hand side of the formula in the statement of the theorem simplifies to $1$ as required.

Next assume that $n$ is the label of the $(m+1)^{th}$ leaf in $\K$, so is the leaf label of subtree $\K_m$. Each of $\K_0, \ldots, \K_m$ contain $s+ij+1$, $0\leq i \leq m$ labels, while there is one extra label in the initial leaf. This implies that $(m+1)^{th}$ leaf has label $n=1+\sum_{i=0}^m (s+ij+1) = p_m$. But from the proof of Theorem \ref{thm:freq} we have that $g(n) = (m+1)j+1$ if $n$ is the label of the $(m+1)^{st}$ leaf in $\K$. We combine these two observations to derive a closed formula for $g(n)$ for such $n$.

Expanding the sum for $n$ we get $n = 1+(m+1)s + \frac{j(m+1)m}{2} + (m+1)$. Multiplying both sides by $j$ and rewriting the result suggestively yields $jn = j(m+1)s + s - s + \frac{(j(m+1)+1-1)(jm+j+1-j-1)}{2} + j(m+1) + 1+j-1$, which is just $jn = sg(n) - s + \frac{(g(n)-1)(g(n)-j-1)}{2} + g(n)+j-1$. Therefore, we have the quadratic equation $g^2(n) + (2s-j)g(n) + (3j-2s-1-2jn) = 0$, which we solve for $g(n)$ in terms of $n$. To do so we need the discriminant of the quadratic $D = (2s-j)^2 + 4(j(2n-3)+2s+1) \geq 0$. But this inequality holds for $n \geq 2$ since all the parameters are non-negative integers. The positive root of the quadratic equation is

\begin{align*}
g(n) = \frac{(j-2s) + \sqrt{(2s-j)^2 + 4(2jn + 2s + 1 - 3j)}}{2}
\end{align*}
which satisfies conclusion of the theorem since $n = p_m = F(n)$ in this case.

Finally we extend this result to all the other nodes in the tree. Between any two successive leaf nodes we know that $g(n)$ is constant and equals the value at the leaf with the largest label which is still smaller than $n$. But as we observed at the outset, this leaf node has label $F(n)$, so $g(n)=g(F(n))$. But $F(F(n))=F(n)$, which completes the proof.

\end{proof}

We conclude by deriving a simple closed form solution for the recursion $g_{1,s,1}(n)$. Of course we could do so by inserting $j=1$ in the preceding formula. What follows is an alternate approach which makes use of the known result (\ref{closed}) for $j=1$ and $s=0$. First we show that we need only examine the case $s<j$.

\begin{theorem}
For $s \geq j$ and $\lambda=1$, if $s=qj+r$ with $0\leq r<j$ and $\alpha=\sum\limits_{i=0}^{q-1}(r+ij+1)$, then $g_{j,s,1}(n) = g_{j,r,1}(n+\alpha) -qj$.
\label{thm:reduce}
\end{theorem}

\begin{proof}
In this proof we will denote the infinite labeled tree $\K$ constructed using parameters $j,s$ and $\lambda=1$ by $\K_{j,s}$. Similarly, we denote the infinite labeled tree $\K$ constructed using parameters $j,r$ and $\lambda=1$ by $\K_{j,r}$. We know that $\K_{j,s}(n)=\K_0, \ldots, \K^*_m$ where $\K_i$, $0\leq i < m$ is a chain having $s+ij+1$ labels with one extra label in the initial leaf of $\K_0$. Since $\K^*_m$ is possible incomplete it contains $t \leq s+mj+1$ labels. Similarly, $\K_{j,r}(n+\alpha) = \K'_0, \ldots, \K'^*_p$. Since each such chain $\K'_i$ has exactly $r+ij+1$ labels (with one extra label on initial leaf of $\K_0$) the first $q$ of these chains have $1+\alpha=1+\sum_{i=0}^{q-1}(r+ij+1)$. The remaining chains $\K'_q, \ldots, \K'^*_p$ need to have $n-1$ labels. Observe that $\K'_{q+i}$, $0 \leq i$ has exactly $r+qj+ij+1=s+ij+1$ labels. If we let  $\K'^*_{q+m}$ be a portion of $\K'_{q+m}$ with $t$ labels on it then $\K'_q, \ldots, \K'^*_{q+m}$ has the same number of labels as $\K_0, \ldots, \K^*_m$ not counting initial leaf of $\K_0$. That is, $\K'_q, \ldots, \K'^*_{q+m}$ has the same number of labels as $\K'_q, \ldots, \K'^*_p$ which shows that $p=q+m$ and $\K'^*_p$ has $t$ labels. This implies that $\K_{j,r}(n+\alpha)$ has $q$ more leaves with weight $j$ than $\K_{j,s}(n)$ and the proof is complete.

\end{proof}

Theorem \ref{thm:reduce} enables us to find a simple closed form solution for the recursion $g_{1,s,1}(n)$. In this case $r=0$ and $\alpha=\sum_{i=1}^s i = \binom{s+1}{2}$. Therefore, $g_{1,s,1}(n) = g_{1,0,1}(n+\binom{s+1}{2}) - s$. Using the closed form solution for $g_{1,0,1}(n)$ given by (\ref{closed}), we write

\begin{align*}
g_{1,s,1}(n) = \left\lfloor \frac{ \left\lfloor \sqrt{8(n+\binom{s+1}{2})} \right\rfloor + 1}{2}\right\rfloor - s.
\end{align*}

\section{Concluding remarks} \label{sec5}

The results of this paper provide the first known example of a combinatorial interpretation for a non-slow solution to a nested recursion, through the use of the new technique of weighted leaf counting. This suggests that other nested recursions with non-slow solutions that occur as special cases of recursion (\ref{recursion}) might also have combinatorial interpretations of a similar sort. This seems to be particularly the case for non-slow monotone solutions, as was the case for the Golomb recursion and its generalization. In subsequent work we plan to investigate several such recursion families.

\end{document}